\documentclass[12pt]{amsart}

\usepackage{float}% table用

\usepackage{tikz-cd}
\usepackage{color}
\usepackage[all]{xy}
\usepackage{mathrsfs}
\usepackage{amsmath, amssymb}
\usepackage{graphics}
\usepackage{graphicx}  
\usepackage{geometry}
\usepackage[cal=boondox,scr=boondoxo]{mathalfa}
\DeclareFontFamily{U}{mathc}{}
\DeclareFontShape{U}{mathc}{m}{it}
{<->s*[1.03] mathc10}{}
\DeclareMathAlphabet{\mathscr}{U}{mathc}{m}{it}

\newtheorem{theo}{Theorem}[section]

\newtheorem{prop}[theo]{Proposition}
\newtheorem{defi}[theo]{Definition}
\newtheorem{rem}[theo]{Remark}
\newtheorem{example}[theo]{Example}

\newtheorem{nota}{Notation}

\newtheorem*{main}{Main Result}

\newcommand{\End}{\operatorname{End}}
\newcommand{\Hom}{\operatorname{Hom}}
\newcommand{\Coker}{\operatorname{Cok}\nolimits}

\newcommand{\id}{\operatorname{id}}

\renewcommand{\mod}{$-$\mathrm{mod}}
\newcommand{\smod}{$-$\underline{\mathrm{mod}}}
\newcommand{\proj}{$-$\mathrm{proj}}
\newcommand{\CM}{\mathrm{CM}}
\newcommand{\sCM}{\underline{\mathrm{CM}}}

\newcommand{\Ext}{\operatorname{Ext}}
\newcommand{\hExt}{\operatorname{\widehat{Ext}}}
\newcommand{\HH}{\operatorname{HH}\nolimits}
\newcommand{\hHH}{\operatorname{\widehat{HH}}}

\newcommand{\Gdim}{\mathrm{G\mbox{-}dim}}

\newcommand{\gldim}{\mathrm{gl.dim}\,}
\newcommand{\injdim}{\mathrm{inj.dim}}
\newcommand{\projdim}{\mathrm{proj.dim}}

\renewcommand{\L}{\Lambda}
\newcommand{\G}{\Gamma}
\newcommand{\Ae}{A^\textrm{e}}
\newcommand{\Le}{\Lambda^\textrm{e}}
\newcommand{\Ge}{\Gamma^\textrm{e}}
\newcommand{\LG}{\Lambda\otimes\Gamma}
\newcommand{\LGe}{(\Lambda\otimes\Gamma)^\textrm{e}}

\newcommand{\T}{\mathcal{T}}
\newcommand{\D}{\mathcal{D}}
\newcommand{\K}{\mathcal{K}}
\newcommand{\APC}{\underline{\mathrm{APC}}}

\numberwithin{equation}{section}

\begin{document}
%
%%%%%%%%%%%%%%%%%%%%%%     Title         %%%%%%%%%%%%%%%%%%%%%%
%
\title[Tate-Hochschild cohomology rings]{Tate-Hochschild cohomology rings for eventually periodic Gorenstein algebras}

%%%%%%%%%%%%%%%%%%%%%%%%%%%%%%%%%%%%
%       author one information
%%%%%%%%%%%%%%%%%%%%%%%%%%%%%%%%%%%%
\author[S. Usui]{Satoshi Usui}
\address{%
Satoshi Usui\\
Department of Mathematics, 
Tokyo University of Science \\ 
1-3 Kagurazaka, Shinjuku-ku, 
Tokyo, 162-8601,
Japan} 
\email{1119702@ed.tus.ac.jp}
%%%%%%%%%%%%%%%%%%%%%%%%%%%%%%%%%%%%
%       author one information
%%%%%%%%%%%%%%%%%%%%%%%%%%%%%%%%%%%%

%\thanks{}

\subjclass[2010]
{16E40, 16E05, 16G50.}

\keywords{Eventually periodic algebras, 
Gorenstein algebras, 
singularity categories, 
Tate-Hochschild cohomology, 
complete resolutions.}

%\date{}

%\dedicatory{}

%%%%%%%%%%%%%%%%%%%%%%%%%%%%%%%%%%%%
%           Abstract
%%%%%%%%%%%%%%%%%%%%%%%%%%%%%%%%%%%%
\begin{abstract}
Tate-Hochschild cohomology of an algebra is a generalization of ordinary Hochschild cohomology, which is defined on positive and negative degrees and has a ring structure. 
Our purpose of this paper is to study the eventual periodicity of an algebra by using the Tate-Hochschild cohomology ring.
First, we deal with eventually periodic algebras and show that they are not necessarily Gorenstein algebras.
Secondly, we characterize the eventual periodicity of a Gorenstein algebra as the existence of an invertible homogeneous element of the Tate-Hochschild cohomology ring of the algebra, which is our main result. 
Finally, we use tensor algebras to establish a way of constructing eventually periodic Gorenstein algebras. 
\end{abstract}
%%%%%%%%%%%%%%%%%%%%%%%%%%%%%%%%%%%%
%           Abstract
%%%%%%%%%%%%%%%%%%%%%%%%%%%%%%%%%%%%

\maketitle
%\tableofcontents

%%%%%%%%%%%%%%%%%%%%%%%%%%%%%%%%%%%%
%           Section ↓
%%%%%%%%%%%%%%%%%%%%%%%%%%%%%%%%%%%%
%\section{}
%%%%%%%%%%%%%%%%%%%%%%%%%%%%%%%%%%%%
%           Section ↑
%%%%%%%%%%%%%%%%%%%%%%%%%%%%%%%%%%%%

%%%%%%%%%%%%%%%%%%%%%%%%%%%%%%%%%%%%
%           Section ↓
%%%%%%%%%%%%%%%%%%%%%%%%%%%%%%%%%%%%
\section{Introduction}
The {\it Tate-Hochschild cohomology} of an algebra was introduced by Wang \cite{Wang_2021} based on the notion of {\it Tate cohomology} defined by Buchweitz \cite{Buch86}. 
It was proved in \cite{Wang_2021} that the Tate-Hochschild cohomology carries a structure of a graded commutative algebra. 
There are studies on the ring structure of the Tate-Hochschild cohomology, such as \cite{DGT19,Nguy12,Usui21,Wang20}.
Recently, Dotsenko, Gélinas and Tamaroff  proved in \cite[Corollary 6.4]{DGT19} that, for a monomial Gorenstein algebra $\L$, the Tate-Hochschild cohomology ring $\hHH^\bullet(\L)$ is isomorphic to $\hHH^{\geq 0}(\L)[\chi^{-1}]$, where $\hHH^{\geq 0}(\Lambda)$ stands for the subring consisting of the non-negative part of $\hHH^{\bullet}(\Lambda)$ and $\chi$ is an invertible homogeneous element of positive degree.
Moreover, the author also showed in \cite[Corollary 3.4]{Usui21} that the same isomorphism holds for a periodic algebra.
In both cases, the invertible element $\chi$ was obtained from the fact that any minimal projective resolution of the given algebra  eventually becomes periodic.

In this paper,  we first deal with eventually periodic algebras (i.e.\ algebras $\L$ with the $n$-th syzygy $\Omega_{\Le}^{n}(\L)$ periodic for some $n \geq 0$).  
It will be revealed that eventually periodic algebras are not necessarily Gorenstein (see Example \ref{example_2}), although it is known that periodic algebras are all Gorenstein.
Secondly, we will study the relationship between the eventual periodicity of a Gorenstein algebra and the Tate-Hochschild cohomology ring of the algebra.
The following is the main result of this paper:

%
%%%%%%%%%%%%%%%      statement     %%%%%%%%%%%%%%%%%%%
%
\begin{main}[see Theorem \ref{claim_6}]
Let $\Lambda$ be a Gorenstein algebra.
Then the following are equivalent.
\begin{enumerate}
    \item $\L$ is an eventually periodic algebra.
    \item The Tate-Hochschild cohomology ring $\hHH^{\bullet}(\L)$ has an invertible homogeneous element of positive degree.
\end{enumerate} 
In this case, there exists an isomorphism $\hHH^{\bullet}(\L) \cong \hHH^{\geq 0}(\L)[\chi^{-1}]$ of graded algebras, where the degree of an invertible homogeneous element $\chi$ equals the period of the periodic syzygy $\Omega_{\Le}^{n}(\L)$ of $\L$ for some $n \geq 0$.
\end{main}
%%%%%%%%%%%%%%%%%%%%%%%%%%%%%%%%%%%%%%%%%%%%%%%%%%%%%%

Our main result requires only the eventual periodicity of a minimal projective resolution of a given Gorenstein algebra.
Hence it turns out that \cite[Corollary 6.4]{DGT19} and \cite[Corollary 3.4]{Usui21} can be obtained from our main result,
because  monomial Gorenstein algebras and periodic algebras are both eventually periodic Gorenstein algebras.
Finally, using tensor algebras, we will provide one of the constructions of eventually periodic Gorenstein algebras.

This paper is organized as follows.
%%%
%%%　　 
%%%
In Section 2, we recall basic facts on  Tate cohomology and Gorenstein algebras.
%%%
%%%　　 
%%%
In Section 3, we give examples of eventually periodic algebras and prove our main result.
%%%
%%%　　 
%%%
In Section 4, we establish a way to construct eventually periodic Gorenstein algebras.

%%%%%%%%%%%%%%%%%%%%%%%%%%%%%%%%%%%%
%           Section ↑
%%%%%%%%%%%%%%%%%%%%%%%%%%%%%%%%%%%%

%%%%%%%%%%%%%%%%%%%%%%%%%%%%%%%%%%%%
%           Section ↓
%%%%%%%%%%%%%%%%%%%%%%%%%%%%%%%%%%%%
\section{Preliminaries}

Throughout this paper, let $k$ be an algebraically closed field.
We write $\otimes_k$ as $\otimes$.
By an algebra $\L$, we mean a finite dimensional associative unital $k$-algebra. 
All modules are assumed to be finitely generated left modules. 
For an algebra $\L$, we denote by $\L\mod$ the category of $\L$-modules, by $\L\proj$ the category of projective $\L$-modules, by $\gldim\L$ the global dimension of $\L$ and by $\Le$ the enveloping algebra $\L \otimes \L^{\rm op}$.
Remark that we can identify $\Le$-modules with $\L$-bimodules.
For a $\L$-module $M$, we denote by $\injdim_{\L}M$ (resp., $\projdim_{\L}M$) the injective (resp., projective) dimension of $M$. 
By a complex $X_\bullet$, we mean a chain complex
\[
X_\bullet\ = \ \cdots \rightarrow X_{i+1} \xrightarrow{d_{i+1}^{X}} X_{i} \rightarrow \cdots.
\]
For a complex $X_\bullet$ and an integer $i$, we denote by $\Omega_{i}(X_\bullet)$ the cokernel $\Coker d^X_{i+1}$ of the differential $d^X_{i+1}$ and by $X_\bullet[i]$ the complex given by $(X_\bullet[i])_j = X_{j-i}$ and  $d^{X[i]} = (-1)^{i}d^{X}$.

%%%%%%%%%%%%%%%%%%%%%%%%%%%%%%%%%%%%%%%%%%%%%%%%%%
\subsection{Tate cohomology rings}
%%%%%%%%%%%%%%%%%%%%%%%%%%%%%%%%%%%%%%%%%%%%%%%%%%

In this subsection, we recall some facts on Tate cohomology rings and Tate-Hochschild cohomology rings.
Let $\L$ be an algebra. Recall that the \textit{singularity category} $\D_{\rm sg}(\Lambda)$ of $\L$ is defined to be the Verdier quotient of the bounded derived category $\D^{\rm b}(\Lambda\mod)$ of $\Lambda\mod$ by the bounded homotopy category $\K^{\rm b}(\Lambda$-$\mathrm{proj})$ of $\Lambda$-$\mathrm{proj}$.
Let $M$ and $N$ be $\L$-modules and $i$ an integer.
Following \cite{Buch86}, we define the \textit{$i$-th Tate cohomology group of  $M$ with coefficients in  $N$} by
\[
    \hExt_{\Lambda}^{i}(M, N)     :=
    \Hom_{\D_{\rm sg}(\Lambda)}(M, N[i]),
\]
where $M$ and $N$ are viewed as complexes concentrated in degree $0$.
We call $\hExt_{\Le}^{i}(\Lambda, \Lambda)$ the \textit{$i$-th Tate-Hochschild cohomology group} of $\Lambda$ and denote it by $\hHH^{i}(\Lambda)$.

Let $\T$ be a triangulated category with shift functor $[1]$. For an object $X$ of $\T$, one can endow $\End_{\T}^{\bullet}(X):= \bigoplus_{i \in \mathbb{Z}} \Hom_{\T}(X, X[i])$ with a structure of a graded ring.
The multiplication is given by the \textit{Yoneda product}
\begin{align*}
     \smile: \Hom_{\T}(X, X[i]) \otimes \Hom_{\T}(X, X[j]) \rightarrow \Hom_{\T}(X, X[i+j]) 
\end{align*}
sending $\alpha \otimes \beta$ to $\alpha [j] \circ \beta.$
If $\T = \D_{\rm sg}(\L)$ and $X = M \in \L\mod$, then we obtain a graded algebra $\hExt_{\Lambda}^{\bullet}(M, M) :=$  $\End_{\D_{\rm sg}(\L)}^{\bullet}(M)$ and call it the \textit{Tate cohomology ring} of $M$, which is called the {\it stabilized Yoneda Ext algebra} of $M$ by Buchweitz \cite{Buch86}.
It was proved by Wang \cite{Wang_2021} 
that the {\it Tate-Hochschild cohomology ring}  $\hHH^{\bullet}(\Lambda) := \hExt_{\Le}^{\bullet}(\Lambda, \Lambda)$ of any algebra $\Lambda$ is a graded commutative algebra.

%%%%%%%%%%%%%%%%%%%%%%%%%%%%%%%%%%%%%%%%%%%%%%%%%%
\subsection{Singularity categories of Gorenstein algebras}
%%%%%%%%%%%%%%%%%%%%%%%%%%%%%%%%%%%%%%%%%%%%%%%%%%
The aim of this subsection is to recall facts on the singularity category of a Gorenstein algebra from \cite{Buch86}.
Let $\L$ be an algebra.
Recall that the {\it stable category} $\L\smod$ of $\L$-modules is the category whose objects are the same as $\L\mod$ and morphisms are given by
\[\underline{\Hom}_{\L}(M, N) := \Hom_{\L}(M, N)/\mathcal{P}(M, N),\]
where $\mathcal{P}(M, N)$ is the space of morphisms factoring through a projective module.
We denote by $[f]$ the element of $\underline{\Hom}_{\Lambda}(M, N)$ represented by a morphism $f: M \rightarrow N$.
There exists a canonical functor $F:\L\smod$ $\rightarrow$ $\D_{\rm sg}(\L)$ making the following square commute:
\[\xymatrix@=15pt{
\L\mbox{-}\mathrm{mod} \ar[r] \ar[d] &  \D^{\rm b}(\L\mbox{-}\mathrm{mod}) \ar[d] \\
\L\mbox{-}\underline{\mathrm{mod}} \ar[r]^-{F} & \D_{\rm sg}(\L)
}\]
where the two vertical functors are the canonical ones, and the upper horizontal functor is the one sending a module $M$ to the complex $M$ concentrated in degree $0$.
Further, the functor $F$ satisfies $F \circ \Omega_{\Lambda} \cong [-1] \circ F$, where $\Omega_{\L}$ is the syzygy functor on $\L\smod$ (i.e.\ the functor sending a module $M$ to the kernel of a projective cover of $M$). 
On the other hand, let $\APC(\L)$ be the homotopy category of acyclic complexes of projective $\L$-modules.
Then taking the cokernel $\Omega_{0}(X_\bullet) = \Coker d_1^X$ of the differential $d_1^X$ for a complex $X_\bullet$ 
defines a functor $\Omega_{0}:\APC(\L) \rightarrow \L\smod$ 
satisfying $\Omega_{0}\circ[-1] \cong \Omega_{\L} \circ \Omega_{0}$.

Recall that an algebra $\L$ is {\it Gorenstein} if $\injdim_{\L}\L < \infty$ and $\injdim_{\L^{\rm op}}\L < \infty$. 
Since the two dimensions coincide (see \cite[Lemma A]{Zaks69}), 
we call a Gorenstein algebra $\L$ with $\injdim_{\L}\L = d$ a {\it $d$-Gorenstein algebra}.
In the rest of this subsection, let $\L$ denote a Gorenstein algebra.
We call a $\L$-module $M$  {\it Cohen-Macaulay} if $\Ext_\L^i(M, \L) = 0$ for all $i > 0$. 
It is clear that projective $\L$-modules are Cohen-Macaulay.\ We denote by $\CM(\L)$ the category of Cohen-Macaulay $\L$-modules. 
It is well-known that $\CM(\L)$ is a Frobenius exact category whose projective objects are precisely projective $\L$-modules, so that the stable category $\sCM(\L)$, the full subcategory of $\L\smod$ consisting of Cohen-Macaulay $\L$-modules, 
carries a structure of a triangulated category (see \cite{Buch86,HappBook}).
In particular, the syzygy functor $\Omega_{\L}$ on $\L\smod$ gives rise to the inverse of the shift functor $\Sigma$ on  $\sCM(\L)$.
We end this subsection with the following result due to Buchweitz.

%
%%%%%%%%%%%%%%%      statement     %%%%%%%%%%%%%%%%%%%
%
\begin{theo}[{\cite[Theorem 4.4.1]{Buch86}}] \label{claim_9}
Let $\L$ be a Gorenstein algebra.
Then there exist equivalences of triangulated categories
\begin{align*}
\xymatrix{\APC(\L) \ar[r]^-{\Omega_0} & \sCM(\L)  \ar[r]^-{\iota_{\L}} &  \D_{\rm sg}(\L),
}
\end{align*}
where the equivalence $\iota_{\L}$ is given by the restriction of $F:\L\smod$ $\rightarrow$ $\D_{\rm sg}(\L)$ to $\sCM(\L)$.
\end{theo} 
%%%%%%%%%%%%%%%%%%%%%%%%%%%%%%%%%%%%%%%%%%%%%%%%%%%%%%

%%%%%%%%%%%%%%%%%%%%%%%%%%%%%%%%%%%%%%%%%%%%%%%%%%
\subsection{Tate cohomology over Gorenstein algebras} \label{sec_1}
%%%%%%%%%%%%%%%%%%%%%%%%%%%%%%%%%%%%%%%%%%%%%%%%%%
This subsection is devoted to recalling another description of Tate cohomology over a Gorenstein algebra. 
Throughout, let $\L$ denote a $d$-Gorenstein algebra unless otherwise stated.
Thanks to Theorem \ref{claim_9}, we can associate to any $\L$-module $M$ an object $T_\bullet = T_\bullet^M$ in $\APC(\L)$, uniquely determined up to isomorphism, satisfying that $\Omega_{0}(T_{\bullet}) \cong M$ in $\D_{\rm sg}(\L)$.
Thus the triangle equivalence $\iota_\L: \sCM(\L)\rightarrow \D_{\rm sg}(\L)$ induces an isomorphism
\[
    \hExt_{\Lambda}^{i}(M, M) \cong \underline{\Hom}_{\L}(\Omega_0(T_\bullet), \Sigma^{i}\Omega_0(T_\bullet)) 
\]
for all $i \in \mathbb{Z}$. 
We identify $\hExt_{\Lambda}^{\bullet}(M, M)$ with $\End_{\sCM(\L)}^{\bullet}(\Omega_{0}(T_\bullet))$  via this isomorphism.

Recall that, for an algebra $\L$, the {\it Gorenstein dimension} $\Gdim_{\L}M$ of a $\L$-module $M$ is defined by the shortest length of a resolution of $M$ by $\L$-modules $X$ with $X \cong X^{**}$ and $\Ext_{\L}^{i}(X, \L) =$ $0$ $=$ $\Ext_{\L^{\rm op}}^{i}(X^{*}, \L)$ for all $i >0$, where we set $(-)^{*}:=\Hom_{\L}(-, \L)$ (see \cite{AusBri69} for its original definition).  
The next proposition is easily obtained from the results in \cite{AvraMart02} applied to the case of Gorenstein algebras:  (1), (2) and (3) follow from  \cite[Theorems 3.1 and 3.2]{AvraMart02},  \cite[Lemma 2.4 and Theorem 3.1]{AvraMart02} and \cite[Theorem 5.2]{AvraMart02}, respectively.

%
%%%%%%%%%%%%%%%      statement     %%%%%%%%%%%%%%%%%%%
%
\begin{prop} \label{claim_12}
The following hold for a module $M$ over a  $d$-Gorenstein algebra $\L$.
\begin{enumerate}
\item  The Gorenstein dimension $\Gdim_{\L}M$ of $M$ satisfies $\Gdim_{\L}M \leq d$ and is equal to the smallest integer $r \geq 0$ for which $\Omega_{\L}^{r}(M)$ is Cohen-Macaulay. 

\item There exists a diagram $T_\bullet \xrightarrow{\theta} P_\bullet \xrightarrow{\varepsilon} M$ satisfying the following conditions:
\begin{enumerate}
    \item[(i)] $T_\bullet \in \APC(\L)$ and $P_\bullet \xrightarrow{\varepsilon} M$ is a projective resolution of $M$.
    
    \item[(ii)] $\theta:T_\bullet \rightarrow P_\bullet$ is a chain map with $\theta_i$  an isomorphism for any $i \gg 0$. 
\end{enumerate}

\item We have that $\Ext_{\Lambda}^{i}(M, M)  \cong \hExt_{\Lambda}^{i}(M, M)$ for all $i > \Gdim_{\L}M$.
\end{enumerate}
\end{prop} 
%%%%%%%%%%%%%%%%%%%%%%%%%%%%%%%%%%%%%%%%%%%%%%%%%%%%%%

We call such a diagram as in Proposition \ref{claim_12} (2) a {\it complete resolution} of $M$ (see \cite{AvraMart02} for its definition in a general setting).
A complete resolution is unique in the sense of \cite[Lemma 5.3]{AvraMart02} (when it exists).

Finally, we explain how we find the corresponding object $T_\bullet^M$ in $\APC(\L)$ for any $\L$-module $M$.
Let $T_\bullet \rightarrow P_\bullet \rightarrow M$ be a complete resolution of $M$.  
Then the complex $T_\bullet$ in $\APC(\L)$ is the object corresponding to $M$ via the triangle equivalence $\iota \circ \Omega_0: \APC(\L) \rightarrow \D_{\rm sg}(\L)$.
Indeed, the morphism $\Omega_0(T_\bullet) \rightarrow M$ induced by the chain map $\theta_{\geq 0}:T_{\geq 0}\rightarrow P_\bullet$ is an isomorphism in $\D_{\rm sg}(\L)$.
Here, $T_{\geq 0}$ stands for the following truncated complex of $T_\bullet$: \[T_{\geq 0} = \cdots \rightarrow T_{2} \xrightarrow{d_{2}^T} T_{1} \xrightarrow{d_{1}^T} T_{0} \rightarrow 0 \rightarrow 0 \rightarrow \cdots.\] 
Thus we conclude that 
constructing a complete resolution of $M$ is equivalent to finding the corresponding object $T_\bullet^M$ of $\APC(\L)$. 
%%%%%%%%%%%%%%%%%%%%%%%%%%%%%%%%%%%%
%           Section ↑
%%%%%%%%%%%%%%%%%%%%%%%%%%%%%%%%%%%%

%%%%%%%%%%%%%%%%%%%%%%%%%%%%%%%%%%%%
%           Section ↓
%%%%%%%%%%%%%%%%%%%%%%%%%%%%%%%%%%%% 
\section{Tate-Hochschild cohomology for eventually periodic Gorenstein algebras} \label{sec_2} 

In this section, we first define eventually periodic algebras and provide examples of them. 
We then prove our main result.

%%%%%%%%%%%%%%%%%%%%%%%%%%%%%%%%%%%%
%           Subsection ↓
%%%%%%%%%%%%%%%%%%%%%%%%%%%%%%%%%%%% 
\subsection{Eventually periodic algebras} \label{sec_3}
As mentioned above, let us first define the eventual periodicity of algebras and provide examples of eventually periodic algebras.

%
%%%%%%%%%%%%%%%      statement     %%%%%%%%%%%%%%%%%%%
%
\begin{defi} \label{def_2} {\rm
Let $\L$ be an algebra. 
A $\L$-module $M$ is called {\it periodic} if $\Omega_{\L}^{p}(M) \cong M$ in $\L\mod$ for some $p > 0$. 
The smallest such $p$ is said to be the {\it period} of $M$. 
We say that $M \in \L\mod$ is {\it eventually periodic} if $\Omega_{\L}^{n}(M)$ is periodic for some $n \geq 0$.
An algebra $\L$ is called {\it periodic} $($resp. {\it eventually periodic$)$} if $\L \in \Le\mod$ is periodic $($resp. eventually periodic$)$.
}\end{defi}
%%%%%%%%%%%%%%%%%%%%%%%%%%%%%%%%%%%%%%%%%%%%%%%%%%%%%%

From the definition, periodic algebras are eventually periodic algebras.
Periodic algebras have been studied for a long time (see \cite{ErdSko08}). 
We know from \cite[Lemma 1.5]{GSS03} that periodic algebras are self-injective algebras (i.e.\ $0$-Gorenstein algebras). 
On the other hand, it follows from the proof of \cite[Corollary 6.4]{DGT19} that monomial Gorenstein algebras are eventually periodic algebras.
It also follows from the formula $\gldim\L = \projdim_{\Le}\L$ (see \cite[Section 1.5]{Happel89}) that algebras of finite global dimension are eventually periodic algebras. 
As will be seen in Example \ref{example_2} below, not all eventually periodic algebras are Gorenstein algebras.

%
%%%%%%%%%%%%%%%      example     %%%%%%%%%%%%%%%%%%%
%
\begin{example} \label{example_2}
{\rm 
\begin{enumerate}
  \setlength{\itemsep}{1mm}

\item  
Let $\L_1$ be  the algebra given by a quiver with relation
\begin{center}
\begin{tikzcd}
1  \arrow[r,shift left=.6ex,"\alpha"] & 2 \arrow[l,shift left =.6ex,"\beta"]
\end{tikzcd} \qquad $\alpha \beta \alpha = 0$.
\end{center}
Then $\L_1$ is a monomial algebra that is not Gorenstein (since $\injdim_{\L}\L e_1$ $=$ $\infty$, where $e_1$ is the primitive idempotent corresponding to the vertex $1$).
Using Bardzell's minimal projective resolution of a monomial algebra (see \cite{Bard97}), we have that $\L_1$ is an eventually periodic algebra having $\Omega_{\L_1^{\rm e}}^{2}(\L_1)$ as its first periodic syzygy.

\item
Let $\L_2$ be the algebra given by a quiver with relation
\begin{center}
\begin{tikzcd}
1 \arrow[out=150,in=210,distance=2em,loop,swap,"\alpha"] \arrow[r,"\beta"] & 2
\end{tikzcd} \qquad $\alpha^2 = 0 $.
\end{center}
Then the algebra $\L_2$ is monomial $1$-Gorenstein and hence eventually periodic.  
Bardzell's minimal projective resolution allows us to see that $\Omega_{\L_2^{\rm e}}^{2}(\L_2)$ is the first periodic syzygy of $\L_2$.

\end{enumerate}
}\end{example}

Moreover, one can see that the algebras in \cite[Example 4.3]{X-WChen_2009} are eventually periodic algebras.  

%%%%%%%%%%%%%%%%%%%%%%%%%%%%%%%%%%%%
%           Subsection ↑
%%%%%%%%%%%%%%%%%%%%%%%%%%%%%%%%%%%%

%%%%%%%%%%%%%%%%%%%%%%%%%%%%%%%%%%%%
%           Subsection ↓
%%%%%%%%%%%%%%%%%%%%%%%%%%%%%%%%%%%% 
\subsection{Main Result}

This subsection is devoted to showing our main result.
We prove it after two propositions below.
Before the first one, we prepare some terminology.  
Recall that we write $\Omega_{i}(X_\bullet) = \Coker d^X_{i+1}$ for a complex $X_\bullet$ and $i \in \mathbb{Z}$.
For a module $M$ over a Gorenstein algebra $\L$, its complete resolution $T_\bullet \rightarrow P_\bullet \rightarrow M$ is called {\it periodic} if there exists an  integer $p >0$ such that $\Omega_i(T_\bullet) \cong \Omega_{i+p}(T_\bullet)$ in $\L\mod$ for all $i \in \mathbb{Z}$.
We call the least such $p$ the {\it period} of the complete resolution.

%
%%%%%%%%%%%%%%%      statement     %%%%%%%%%%%%%%%%%%%
%
\begin{prop} \label{claim_11}
Let $\L$ be a Gorenstein algebra and $M$ a $\L$-module.  
If there exists an integer $n \geq 0$ such that $\Omega_{\L}^{n}(M)$ is periodic of period $p$, then $M$ admits a periodic complete resolution of period $p$. 
Further, the period of the periodic complete resolution is independent of the choice of periodic syzygies.
\end{prop}
\begin{proof} 
Assume that there exists a minimal projective resolution $P_\bullet \rightarrow M$ satisfying that $\Omega_\L^{n}(M)$ is periodic of period $p$.   
Then, by using the periodicity of $\Omega_\L^n(M)$, we can extend the truncated complex  $P_{\geq n}$ to an (unbounded) complex $T_{\bullet}$ in $\APC(\L)$ having the following properties: 
\begin{enumerate}
    \item[(i)] $T_{\geq n} = P_{\geq n}$.
    \item[(ii)] For each $i \in \mathbb{Z}$, there exists an integer $0 \leq j < p$ such that $\Omega_{i}(T_\bullet) \cong \Omega_\L^{n+j}(M)$.
\end{enumerate}
In particular, one sees that $\Omega_{i}(T_\bullet) \cong \Omega_{i+p}(T_\bullet)$ for all $i \in \mathbb{Z}$. 
Note that one may take $T_\bullet = 0$ if $\projdim_\L M < \infty$.
It follows from Theorem \ref{claim_9} that $\Omega_{i}(T_\bullet) = \Sigma^{-i} \Omega_{0}(T_\bullet)$ is Cohen-Macaulay for each $i \in \mathbb{Z}$, where $\Sigma$ denotes the shift functor on $\sCM(\L)$.
Then it is easily checked that $\Hom_{\K(\L)}(T_\bullet, \L[i])= 0$ for all $i \in \mathbb{Z}$, where $\K(\L)$ is the homotopy category of $\L$-modules.
Hence, as in \cite[Lemma 2.4]{CorKro97}, the family $\{\id_{T_j}\}_{j \geq n}$ can be extended uniquely up to homotopy to a chain map $\theta: T_{\bullet} \rightarrow P_{\bullet}$ with $\theta_{j}$ the identity for all $j \geq n$.
Therefore, the chain map $\theta$ gives rise to the desired complete resolution.
We remark that the period of the resulting complete resolution does not depend on the choice of $n$.
Indeed, if we take the smallest integer $r \geq 0$ such that $\Omega_{\L}^{r}(M)$ is periodic, then, for each $i \geq n$, the module $\Omega_{\L}^{i}(M)$ is periodic and has the same period as $\Omega_{\L}^{r}(M)$.
\end{proof}
%%%%%%%%%%%%%%%%%%%%%%%%%%%%%%%%%%%%%%%%%%%%%%%%%%%%%%

Recall that the Yoneda product of the Tate cohomology ring $\hExt_{\L}^{\bullet}(M, M)$ is denoted by  $\smile$.

%
%%%%%%%%%%%%%%%      statement     %%%%%%%%%%%%%%%%%%%
%
\begin{prop} \label{claim_5}  
Let $\L$ be a Gorenstein algebra and $M$ a $\L$-module.  
Then the following are equivalent.
\begin{enumerate}
    \item $M$ is eventually periodic.
    \item The Tate cohomology ring $\hExt_{\L}^{\bullet}(M, M)$ has an invertible homogeneous element of positive degree.
\end{enumerate} 
\end{prop}
\begin{proof} 
It suffices to prove the statement for $M \in \L\mod$ with $\projdim_\L M =\infty$.
First, we assume that a $\L$-module $M$ satisfies that $\Omega_{\L}^{n}(M)$ is periodic of period $p$ for some $n \geq 0$.
By Proposition \ref{claim_11}, there exists a complete resolution $T_{\bullet} \rightarrow P_{\bullet} \rightarrow M$ such that $\Omega_{0}(T_\bullet)$ is periodic of period $p$, where $p$ is the period of $\Omega_\L^{n}(M)$. 
We fix this complete resolution.  
Then the shift functor $\Sigma$ on $\sCM(\L)$ satisfies $\Sigma^{i} \Omega_{0}(T_\bullet) = \Omega_{-i}(T_\bullet)$ for all $i \in \mathbb{Z}$. 
Let $f \in \Hom_{\L}(\Omega_{p}(T_\bullet),  \Omega_{0}(T_\bullet))$ be an isomorphism and consider two homogeneous elements
\begin{align*}
    x:= \Sigma^{p}[f] \in \hExt_{\Lambda}^{p}(M, M) 
    \quad \mbox{and} \quad 
    y:= [f^{-1}] \in \hExt_{\Lambda}^{-p}(M, M). 
\end{align*} 
Then we have $x \smile y = (\Sigma^{-p}x) \circ y = [f] \circ [f^{-1}] = 1$  and similarly $y \smile x = 1$, where we set $1 := [\id_{\Omega_{0}(T_\bullet)}]$. 

Conversely, we let $T_\bullet \rightarrow P_\bullet \rightarrow M$ be a complete resolution of $M$ and assume that there exists an isomorphism 
\[
x \in \underline{\Hom}_{\L}(\Omega_0(T_\bullet), \Sigma^{p}\Omega_0(T_\bullet)) = \hExt_{\L}^{p}(M, M)
\]
of degree $p >0$.
From the definition of complete resolutions, we have  
\begin{align*}
\underline{\Hom}_{\L}(\Omega_0(T_\bullet), \Sigma^{p}\Omega_0(T_\bullet))
&\cong
\underline{\Hom}_{\L}(\Sigma^{-m-p}\Omega_0(T_\bullet), \Sigma^{-m}\Omega_0(T_\bullet)) \\
&\cong
\underline{\Hom}_{\L}(\Omega_{\L}^{m+p}(M), \Omega_{\L}^{p}(M))
\end{align*}
for some sufficiently large $m >0$.
Hence we get $\Omega_{\L}^{m+p}(M) \cong \Omega_{\L}^{m}(M)$ in $\L\smod$.
This implies that $\Omega_{\L}^{m+p}(M) \oplus P \cong \Omega_{\L}^{m}(M)\oplus Q$ in $\L\mod$ for some $P$ and $Q \in \L\proj$.
By applying the syzygy functor $\Omega_{\L}$ to this isomorphism, we obtain an isomorphism $\Omega_{\L}^{m+p+1}(M) \cong \Omega_{\L}^{m+1}(M) $ in $\L\mod$. 
This completes the proof.
\end{proof}
%%%%%%%%%%%%%%%%%%%%%%%%%%%%%%%%%%%%%%%%%%%%%%%%%%%%%%

Using Proposition \ref{claim_5}, we obtain our main result.

%
%%%%%%%%%%%%%%%      statement     %%%%%%%%%%%%%%%%%%%
%
 
\begin{theo} \label{claim_6}
Let $\Lambda$ be a Gorenstein algebra.
Then the following are equivalent.
\begin{enumerate}
    \item $\L$ is an eventually periodic algebra.
    \item The Tate-Hochschild cohomology ring $\hHH^{\bullet}(\L)$ has an invertible homogeneous element of positive degree.
\end{enumerate} 
In this case, there exists an isomorphism $\hHH^{\bullet}(\L) \cong \hHH^{\geq 0}(\L)[\chi^{-1}]$ of graded algebras, where the degree of an invertible homogeneous element $\chi$ equals the period of the periodic syzygy $\Omega_{\Le}^{n}(\L)$ of $\L$ for some $n \geq 0$.
\end{theo}

\begin{proof} 
We know from \cite[Proposition 2.2]{AusRei91} that if $\L$ is a Gorenstein algebra, then so is the enveloping algebra $\Le$. 
Hence the former statement follows from Proposition \ref{claim_5} applied to $\L \in \Le\mod$. 
On the other hand, suppose that the Gorenstein algebra $\L$ satisfies that  $\Omega_{\Le}^{n}(\L)$ is periodic for some $n \geq 0$. 
By the proof of Proposition \ref{claim_5}, there exists an invertible homogeneous element $\chi \in \hHH^{\bullet}(\L)$ whose degree equals the period of the periodic $\Le$-module $\Omega_{\Le}^n(\L)$.
Then the fact that $\hHH^{\bullet}(\L)$ is a graded commutative algebra yields the desired isomorphism of graded algebras (cf.\ the proof of \cite[Corollary 3.4]{Usui21}).
\end{proof}
%%%%%%%%%%%%%%%%%%%%%%%%%%%%%%%%%%%%%%%%%%%%%%%%%%%%%%

We end this subsection with the following three remarks.

%
%%%%%%%%%%%%%%%      remark     %%%%%%%%%%%%%%%%%%%
%
\begin{rem}{\rm
From the definition of singularity categories, an algebra $\L$ has finite projective dimension as a $\Le$-module if and only if its Tate-Hochschild cohomology ring is the zero ring (cf.\ \cite[Section 1]{Buch86}).
Thus Theorem \ref{claim_6} makes essential sense for the case of Gorenstein algebras with infinite global dimension.
}\end{rem}
%%%%%%%%%%%%%%%%%%%%%%%%%%%%%%%%%%%%%%%%%%%%%%%%%%%%%%

%
%%%%%%%%%%%%%%%      remark     %%%%%%%%%%%%%%%%%%%
%
\begin{rem}{\rm
Applying Theorem \ref{claim_6} to  monomial Gorenstein algebras and to periodic algebras, one obtains  \cite[Corollary 6.4]{DGT19} and \cite[Corollary 3.4]{Usui21}, respectively.
}\end{rem}
%%%%%%%%%%%%%%%%%%%%%%%%%%%%%%%%%%%%%%%%%%%%%%%%%%%%%%

%
%%%%%%%%%%%%%%%      remark     %%%%%%%%%%%%%%%%%%%
%
\begin{rem}{\rm  
For an eventually periodic Gorenstein algebra $\L$, 
one can obtain all of the $\dim_{k}\hHH^{*}(\L)$ by using Theorem \ref{claim_6} and the Hochschild cohomology $\HH^\bullet(\L):= \bigoplus_{i\geq 0}\Ext_{\Le}^i(\Lambda, \Lambda)$ of $\L$ (see Example \ref{example_1}).  
However, it is still open how we compute  the ring structure of $\hHH^{\geq 0}(\L)$ (cf.\,\cite[Proposition 3.7]{Usui21}). 
}\end{rem}
%%%%%%%%%%%%%%%%%%%%%%%%%%%%%%%%%%%%%%%%%%%%%%%%%%%%%%

%%%%%%%%%%%%%%%%%%%%%%%%%%%%%%%%%%%%
%           Subsection ↑
%%%%%%%%%%%%%%%%%%%%%%%%%%%%%%%%%%%%

%%%%%%%%%%%%%%%%%%%%%%%%%%%%%%%%%%%%
%           Section ↑
%%%%%%%%%%%%%%%%%%%%%%%%%%%%%%%%%%%%

%%%%%%%%%%%%%%%%%%%%%%%%%%%%%%%%%%%%
%           Section ↓
%%%%%%%%%%%%%%%%%%%%%%%%%%%%%%%%%%%%
\section{Construction of eventually periodic Gorenstein algebras} \label{sec_4}
In this section, we aim at establishing a way of constructing eventually periodic Gorenstein algebras.
First, we show two propositions which will be used latter.
Let us start with the following.

%
%%%%%%%%%%%%%%%      statement     %%%%%%%%%%%%%%%%%%%
%
\begin{prop} \label{claim_1} 
Any periodic $\L$-module $M$ over a $d$-Gorenstein algebra $\L$ is Cohen-Macaulay.
\end{prop}
\begin{proof} 
Assume that $M$ is a periodic $\L$-module of period $p$.
Since $\Omega_{\L}^{i}(M) \in \CM(\L)$ for $i \geq d$ by \cite[Lemma 4.2.2]{Buch86}, we have that $M \cong \Omega_{\L}^{j p}(M) \in \CM(\L)$ for some $j \gg 0$.
\end{proof}
%%%%%%%%%%%%%%%%%%%%%%%%%%%%%%%%%%%%%%%%%%%%%%%%%%%%%%

We now show that, for an eventually periodic Gorenstein algebra $\L$, the smallest integer $n \geq 0$ satisfying that $\Omega_{\Le}^{n}(\L)$ is periodic has a lower bound.

%
%%%%%%%%%%%%%%%      statement     %%%%%%%%%%%%%%%%%%%
%
\begin{prop} \label{claim_4}
Let $\L$ be a $d$-Gorenstein algebra. 
Assume that there exists an integer $n \geq 0$ such that $\Omega_{\Le}^{n}(\L)$ is periodic.
Then the least such integer $n$ satisfies $n \geq d$.
In particular, an equality holds if and only if there exists a simple $\L$-module $S$ such that $\Ext_{\Lambda}^{n}(S, \L) \not= 0$.
\end{prop}
\begin{proof}
Let $\L$ be an eventually periodic Gorenstein algebra and $P_\bullet \rightarrow \L$ a minimal projective resolution of $\L$ over $\Le$ satisfying that $\Omega_{\Le}^{n}(\L)$ is the first periodic syzygy of period $p$.
For any $M \in \L\mod$, an exact sequence $P_\bullet \otimes_\L M \rightarrow \L \otimes_\L M = M$ is a projective resolution of $M$ and has the property that
$\Omega_{n}(P_\bullet \otimes_\L M) = \Omega_{\Le}^{n}(\L) \otimes_\L M \cong \Omega_{\Le}^{n+i p}(\L) \otimes_\L M= \Omega_{n+i p}(P_\bullet \otimes_\L M)$ 
for all $i \geq 0$.
In particular, as in Proposition \ref{claim_1}, one concludes that $\Omega_{n}(P_\bullet \otimes_\L M)$ is Cohen-Macaulay.  
This implies that $n \geq \injdim_{\L}\L = d$.
Indeed, for any $\L$-module $M$, we have 
$\Ext_{\Lambda}^{n+1}(M, \L) \cong \Ext_{\Lambda}^{1}(\Omega_{n}(P_\bullet \otimes_\L M), \L) = 0. $

For the latter statement, we first suppose that $n=d$.
Then it follows from \cite[Proposition 2.4]{DGT19} that we have $n=\Gdim_{\L} (\L/\mathfrak{r})$, where $\mathfrak{r}$ denotes the Jacobson radical of $\L$. 
This shows that $\Ext_{\Lambda}^{n}(\L/\mathfrak{r}, \L) \not = 0$, so that one obtains the desired simple $\L$-module.
Conversely, assume that $\Ext_{\Lambda}^{n}(S, \L) \not= 0$ for some simple $\L$-module $S$.
Then one concludes that $\Omega_{\L}^{n-1}(S) \not\in \CM(\L)$.
However, since we know that $\Omega_{\L}^{n}(S)$ is Cohen-Macaulay, we have $n= \Gdim_{\L}S$ and hence $n \leq d$.  
Then the proof is completed since $n \geq d$ by the former statement.
\end{proof}
%%%%%%%%%%%%%%%%%%%%%%%%%%%%%%%%%%%%%%%%%%%%%%%%%%%%%%

Now, we recall some facts on projective resolutions for tensor algebras.
Let $\L$ and $\G$ be algebras and $P_\bullet \xrightarrow{\varepsilon_{\L}} \L$ and $Q_\bullet \xrightarrow{\varepsilon_{\G}} \G$ projective resolutions as bimodules.
Then the tensor product $P_\bullet \otimes Q_\bullet  \xrightarrow{\varepsilon_{\L} \otimes\, \varepsilon_{\G}} \L \otimes \G$ is a projective resolution of the tensor algebra $\L \otimes \G$ over $\LGe$ (see \cite[Section X.7]{Maclane}). 
Here, we identify $\LGe$ with $\Le \otimes \Ge$.  
It also follows from \cite[Lemma 6.2]{benson_iyengar_krause_pevtsova_2020} that if both $P_\bullet \rightarrow \L$ and $Q_\bullet \rightarrow \G$ are minimal, then so is $P_\bullet \otimes Q_\bullet \rightarrow \LG$.

From now on, we assume that $\L$ is a periodic algebra of period $p$ and that $\G$ is an algebra of finite global dimension $n$. 
Set  $A := \LG$.
Since periodic algebras are self-injective algebras, it follows from \cite[Lemma 6.1]{benson_iyengar_krause_pevtsova_2020} that we have 
\[\injdim\,A = \injdim\,\L + \injdim\,\G = 0 + n = n \]
as one-sided modules.
Thus $A$ is an $n$-Gorenstein algebra.  
Note that the same lemma also implies that the enveloping algebra $A^{\rm e}$ is a $(2n)$-Gorenstein algebra.
We now show that the algebra $A$ has an eventually periodic minimal projective resolution.

%
%%%%%%%%%%%%%%%      statement     %%%%%%%%%%%%%%%%%%%
%
\begin{prop} \label{claim_7} 
Let $\L$ and $\G$ be as above. 
Then $A = \LG$ is an eventually periodic $n$-Gorenstein algebra having $\Omega_{\Ae}^{n}(A)$ as its first periodic syzygy. 
\end{prop}
\begin{proof} 
Let $P_\bullet \rightarrow \L$ and $Q_\bullet \rightarrow \G$ be minimal projective resolutions as bimodules.
Recall that the $r$-th component of the total complex $P_\bullet \otimes Q_\bullet$ with $r \geq 0$ is given by
\begin{align*}
    (P_\bullet \otimes Q_\bullet)_r = \bigoplus_{i = 0}^{r} P_{r-i} \otimes Q_{i}.
\end{align*}
Since $Q_{i} = 0$ for $i > n =\gldim\G = \projdim_{\Ge}\G$, we have
\begin{align*}
    (P_\bullet \otimes Q_\bullet)_r = \bigoplus_{i = 0}^{n} P_{r-i} \otimes Q_{i}
\end{align*}
for all $r \geq n$. 
Moreover, the $(r+1)$-th differential 
\[d_{r+1}^{P\otimes Q}: (P_\bullet \otimes Q_\bullet)_{r+1} \rightarrow (P_\bullet \otimes Q_\bullet)_{r} \quad \mbox{ ($r \geq n$)}\]   
can be written as the square matrix $(\partial^{ij}_{r+1})_{ij}$ of degree $n+1$ whose $(i, j)$-th entry
\[
\partial^{ij}_{r+1}: P_{r+1-(j-1)} \otimes Q_{j-1} \rightarrow P_{r-(i-1)} \otimes Q_{i-1} \quad (1 \leq i, j \leq n+1)
\] 
is given by 
\[
\partial^{ij}_{r+1} 
=\begin{cases} 
d_{r-i+2}^{P} \otimes \id_{Q_{i-1}} & \quad \mbox{if $i =j$;}\\[1mm]
(-1)^{r-i+1} \id_{P_{r-i+1}} \otimes\,d_{i}^{Q} & \quad \mbox{if $j = i+1$;}\\[1mm]
0 & \quad \mbox{otherwise.}
\end{cases}
\] 
We claim that $\Coker d_{n+p+1}^{P\otimes Q} \cong \Coker d_{n+1}^{P\otimes Q}$.
First, suppose that $p$ is even. 
Since $\partial^{ij}_{n+p+1} = \partial^{ij}_{n+1}$ for all  $1\leq i, j \leq n+1$ 
because $p$ is even and  $d_{l}^P = d_{l+p}^P$ for any $l \geq 0$,  we conclude that $d_{n+p+1}^{P\otimes Q} = d_{n+1}^{P\otimes Q}$, which implies the claim.
Now, assume that $p$ is odd.
Consider the isomorphism of $\Ae$-modules between $(P_\bullet \otimes Q_\bullet)_{r}$ and $(P_\bullet \otimes Q_\bullet)_{r+p}$ with $r \geq n$ induced by the diagonal matrix $D$ of degree $n+1$ whose $(i,i)$-th entry is $(-1)^{n+i}$. 
Together with the fact that $p+1$ is even, a direct calculation shows that there exists a commutative diagram of $\Ae$-modules with exact rows
\[\xymatrix{
(P_\bullet\otimes Q_\bullet)_{n+p+1} \ar[r]^-{d_{n+p+1}^{P \otimes Q}} \ar[d]^-{D}_-{\cong} & (P_\bullet \otimes Q_\bullet)_{n+p} \ar[r] \ar[d]^-{D}_-{\cong} & \Coker d_{n+p+1}^{P \otimes Q} \ar[r] & 0\\
(P_\bullet \otimes Q_\bullet)_{n+1} \ar[r]^-{d_{n+1}^{P\otimes Q}}  & (P_\bullet \otimes Q_\bullet)_{n} \ar[r]  & \Coker d_{n+1}^{P \otimes Q} \ar[r] & 0
}\]
This implies the claim.
Since the projective resolution $P_\bullet \otimes Q_\bullet \rightarrow A$ is minimal, we have that
$
\Omega_{\Ae}^{n+p}(A) = \Coker d_{n+p+1}^{P\otimes Q} \cong \Coker d_{n+1}^{P\otimes Q} = \Omega_{\Ae}^{n}(A)$. 
From Proposition \ref{claim_4} and the isomorphism, we see that the $n$-th syzygy $\Omega_{\Ae}^{n}(A)$ is the first periodic syzygy of $A$. 
\end{proof}
%%%%%%%%%%%%%%%%%%%%%%%%%%%%%%%%%%%%%%%%%%%%%%%%%%%%%%

%
%%%%%%%%%%%%%%%      remark     %%%%%%%%%%%%%%%%%%%
%
\begin{rem} \label{rem_1} {\rm 
Proposition \ref{claim_12} allows us to get $\Gdim_{\Ae}A \leq 2n= \injdim_{\Ae}\Ae$ and hence $\HH^i(A) \cong \hHH^i(A)$ for all $i > 2n$. 
On the other hand, the $i$-th syzygy $\Omega_{\Ae}^{i}(A)$ of $A$ is Cohen-Macaulay for any $i \geq n $ by Propositions \ref{claim_1} and \ref{claim_7}.
Again, Proposition \ref{claim_12} yields that $\Gdim_{\Ae}A \leq n$.
One of the advantages of this observation is that there exists an isomorphism $\HH^i(A) \cong \hHH^i(A)$ for all $i > n$. 
}\end{rem}
%%%%%%%%%%%%%%%%%%%%%%%%%%%%%%%%%%%%%%%%%%%%%%%%%%%%%%

%
%%%%%%%%%%%%%%%      remark     %%%%%%%%%%%%%%%%%%%
%
\begin{rem} \label{rem_2} {\rm
It follows from Theorem \ref{claim_6} and the proof of Proposition \ref{claim_7}  that  the Tate-Hochschild cohomology ring $\hHH^{\bullet}(A)$ of $A$ is of the form $\hHH^{\geq 0}(A)[\chi^{-1}]$, where the degree of $\chi$ divides the period $p$ of $\L$.
We hope to address the degree of $\chi$ in a future paper. 
}\end{rem}
%%%%%%%%%%%%%%%%%%%%%%%%%%%%%%%%%%%%%%%%%%%%%%%%%%%%%%

We end this section with the following two examples. 
Note that the tensor algebra $A$ in Example \ref{example_1} can be found in \cite[Example 6.3]{benson_iyengar_krause_pevtsova_2020}.

%
%%%%%%%%%%%%%%%      example     %%%%%%%%%%%%%%%%%%%
%
\begin{example} \label{example_3} {\rm % example 1
For an integer $n \geq 0$, let $\G_n$ be the algebra given by a quiver with relations
\begin{center}
\begin{tikzcd}
0  \arrow[r,"\alpha_0"] & 1 \arrow[r] & \cdots \arrow[r] & n-1 \arrow[r,"\alpha_{n-1}"] & n
\end{tikzcd}
\quad $\alpha_{i+1} \alpha_{i}= 0$ for  $i = 0, \ldots, n-2$.
\end{center}
Then we have $\gldim\G_n = n$.
By Proposition \ref{claim_7}, any periodic algebra $\L$ gives us an eventually periodic $n$-Gorenstein algebra $A = \L \otimes \G_n$ with $\Omega_{A^{\rm e}}^{n}(A)$ the first periodic syzygy of $A$.

}\end{example}
%%%%%%%%%%%%%%%%%%%%%%%%%%%%%%%%%%%%%%%%%%%%%%%%%%%%%%

%
%%%%%%%%%%%%%%%      example     %%%%%%%%%%%%%%%%%%%
%
\begin{example} \label{example_1} {\rm % example 1 
Let $\L =k[x]/(x^2)$ and let $\G$ be the algebra $\G_1$ defined in Example \ref{example_3}.
Thanks to Bardzell's minimal projective resolution, we see that $\L$ is a periodic algebra whose period is equal to $1$ if $\mathrm{char}\,k =2$ and to $2$ otherwise.
On the other hand, the tensor algebra $A = \LG$ is given by the following quiver with relations 
\begin{center}
\begin{tikzcd}
1 \arrow[out=150,in=210,distance=2em,loop,swap,"\alpha"] \arrow[r,"\beta"] & 2 \arrow[in=-30,out=30,distance=2em,loop,"\gamma"]
\end{tikzcd} \qquad \quad  $\alpha^2 =0 = \gamma^2 \quad \mbox{and} \quad \beta \alpha  =  \gamma \beta$.
\end{center} 
Thus we see that $A$ is a (non-monomial) eventually periodic Gorenstein algebra whose first periodic syzygy is $\Omega_{A^{\rm e}}^{1}(A)$. 
Now, we compute $\dim_{k} \hHH^i(A)$ for all $i \in \mathbb{Z}$. 
It follows from \cite[Section 1.6]{Happel89} that the Hochschild cohomology ring $\HH^\bullet(\G)$ is of the form \[\HH^\bullet(\G) = k.\]
According to \cite[Section 5]{BLM2000}, the Hochschild cohomology ring $\HH^\bullet(\L)$ is as follows:
\[
\HH^\bullet(\L) =
\begin{cases} 
k[a_0,\,  a_1]/(a_0^2) & \mbox{if\ } \mathrm{char}\,k = 2;\\[2mm]
k[a_0, \, a_1,\,  a_2]/(a_0^2,\, a_1^2,\, a_0 a_1,\, a_0 a_2 ) & \mbox{if\ } \mathrm{char}\,k \not = 2,
\end{cases} 
\] 
where the index $i$ of a homogeneous element $a_i$ denotes the degree of $a_i$. 
On the other hand, by \cite[Lemma 3.1]{LeZhou14},  there exists an isomorphism of graded algebras
\begin{align*}
    \HH^\bullet(A) \cong \HH^\bullet(\L) \otimes \HH^\bullet(\G) =\HH^\bullet(\L).
\end{align*} 
It follows from the first remark after Proposition \ref{claim_7} that $\HH^i(A) \cong \hHH^i(A)$ for all $i > 1$.
Hence, the fact that $\hHH^*(A) \cong \hHH^{*+p}(A)$ with $p$ the period of $\L$  
(see  the second remark after Proposition \ref{claim_7}) implies that, for any integer $i$, we have
\[
\dim_{k} \hHH^i(A) =
\begin{cases} 
2 & \mbox{if\ } \mathrm{char}\,k = 2;\\[1mm]
1 & \mbox{if\ } \mathrm{char}\,k \not = 2.
\end{cases} 
\] 

}\end{example}
%%%%%%%%%%%%%%%%%%%%%%%%%%%%%%%%%%%%%%%%%%%%%%%%%%%%%%

%%%%%%%%%%%%%%%%%%%%%%%%%%%%%%%%%%%%
%           Section ↑
%%%%%%%%%%%%%%%%%%%%%%%%%%%%%%%%%%%%

%%%%%%%%%%%%%%%%%%%%%%%%%%%%%%%%%%%%
%           Acknowledgments
%%%%%%%%%%%%%%%%%%%%%%%%%%%%%%%%%%%%
\section*{Acknowledgments} 
The author would like to express his appreciation to the referee(s) for valuable suggestions and comments and for pointing out an error in the manuscript.
The author also would like to thank Professor Katsunori Sanada, Professor Ayako Itaba and Professor Tomohiro Itagaki for their tremendous support for the improvement of the manuscript of the paper.
%%%%%%%%%%%%%%%%%%%%%%%%%%%%%%%%%%%%
%           Acknowledgments
%%%%%%%%%%%%%%%%%%%%%%%%%%%%%%%%%%%%

\bibliographystyle{plain}
\bibliography{ref}

\end{document}